\documentclass[a4paper,11pt]{article}

\title{Generation of Iterated Wreath Products Constructed from
  Full Transformation Monoids and Symmetric Groups}
\author{Jiaping Lu \\[10pt]
  Mathematical Institute, University of St Andrews,\\
  North Haugh, St Andrews, Fife, KY16 9SS\\[10pt]
  \begin{tabular}{c}\texttt{jl337@st-andrews.ac.uk}\\
    \texttt{jiapingljp@hotmail.com}
  \end{tabular}}

\usepackage{amsmath,amssymb}
\usepackage[margin=2.9cm]{geometry}
\usepackage[amsmath,thmmarks]{ntheorem}
\usepackage{enumitem}
\usepackage{leftindex}
\usepackage{tikz}
\usepackage{mathtools}
\usepackage{tikz}
\usepackage{booktabs}

\theorembodyfont{\slshape}
\newtheorem{lem}{Lemma}[section]

\newtheorem{cor}[lem]{Corollary}
\newtheorem{thm}[lem]{Theorem}
\theorembodyfont{\normalfont}

\theorembodyfont{\slshape}
\theoremnumbering{Alph}

\makeatletter
\theoremstyle{nonumberplain}
\theorembodyfont{\normalfont}
\theoremheaderfont{\normalfont\scshape}
\theoremsymbol{~\ensuremath\square}
\newtheoremstyle{proofstyle}%
  {\item[\theorem@headerfont\hskip\labelsep ##1\theorem@separator]}%
  {\item[\theorem@headerfont\hskip\labelsep ##3\theorem@separator]}
\theoremstyle{proofstyle}
\newtheorem{prf}{Proof:}
\makeatother

\renewcommand{\leq}{\leqslant}
\renewcommand{\geq}{\geqslant}
\newcommand{\nbd}{\nobreakdash-}

\newcommand{\End}[1]{\operatorname{End}(#1)}

\newcommand{\order}[1]{\mathopen{|}#1\mathclose{|}}

\newcommand{\rank}{\operatorname{rank}}
\newcommand{\set}[2]{\{\,#1\mid#2\,\}}

\renewcommand{\geq}{\geqslant}
\renewcommand{\leq}{\leqslant}

\renewcommand{\wr}{\operatorname{wr}}

\setlist[enumerate,1]{label={\normalfont(\roman*)}}

\usepackage{hyperref}
\hypersetup{colorlinks=true, linkcolor={red!50!black},
  citecolor={green!50!black}}

\begin{document}

\maketitle

\begin{abstract}
  The rank of a semigroup is the cardinality of a smallest generating set. In this paper, we introduce iterated wreath products of transformation semigroups. We determine the rank of iterated wreath products of full transformation monoids and symmetric groups.
\end{abstract}


\section{introduction}

Let $S$ be a semigroup and let $U$ be a subset of $S$. We say $U$ \textit{generates} $S$ if every element of $S$ can be expressed as a word using the elements of $U$. If $S$ is furthermore a monoid, then by convention, we take the identity as an empty word. The \textit{rank} of a semigroup $S$, denoted by $\rank(S)$, is defined to be the minimum cardinality of the generating sets of $S$. It is often taught in semigroup courses that  most full transformation monoids require $3$ elements to generate
 and that most full partial transformation monoids need $4$ (see \cite[Exercises 1.9.7 and 1.9.13 respectively]{Howie}). There has been extensive research on the rank of semigroup; see, for example, \cite{GomesHowie1992,GomesHowie1987, HowieMcFadden}.

Apart from the rank of a semigroup, given a semigroup $S$ and a subset $U$ of $S$, we have the \textit{relative rank} of $S$ modulo $U$, denoted by $\rank(S:U)$. It is defined as follows:
\[
\rank(S:U):=\set{\min\order{X}}{\langle U\cup X\rangle=S, X\subseteq S}.
\]

Our research is inspired by results on transformation monoids preserving a uniform partition. Pei~\cite{Pei} establishes that transformations preserving an identical uniform partition in a full transformation monoid form a submonoid of rank at most $6$. Ara\'ujo and Schneider~\cite{wrSemi} improve Pei's results. Using notions of wreath products, they prove that such submonoids can be expressed as wreath products of full transformation monoids~\cite[Lemma 2.1]{wrSemi}, and they show that such wreath products have rank $4$. Furthermore, in~\cite{Araujo2015}, Ara\'ujo et al.\ investigated the rank of transformation semigroups which preserves an arbitrary partition.

The results in~\cite{Araujo2015, wrSemi} rely upon the relative rank. 
In particular, Ara\'ujo and Schneider's results in~\cite{wrSemi} depend on the relative rank modulo a wreath product of symmetric groups and the rank of this group, the latter of which is determined using representation theory.
Mitchell and East~\cite{EastMitchell} later used an elementary method to establish the rank of this wreath product of permutation groups. 
The author of this paper extended their result and determined the minimum number of generator of iterated wreath products of symmetric groups ~\cite[Theorem A \& Corollary B]{LuQuick}. 
Using this theorem and the relative rank, we will establish iterated wreath product of transformation semigroups and determine the rank of iterated wreath products of full transformation monoids and symmetric groups, extending Ara\'ujo and Schneider's result in~\cite{wrSemi}.

We state our theorem of this paper below.

\begin{thm}\label{thm:main}
Let $n \geq 2$ and $(T_1, X_1)$,~$(T_{2}, X_2)$, \dots,~$(T_{n}, X_n)$ be transformation monoids each of which is either a full transformation monoid or a symmetric group on a finite non-singleton set.  
Set $\Omega_1=X_1$ and $V_1=T_1$. 
For $2\leq i\leq n$, let $\Omega_{i}=X_i\times \Omega_{i-1}$ and let $V_i =  T_{i} \wr V_{i-1}$ be the iterated wreath product of $T_i$ by the transformation monoid $(V_{i-1}, \Omega_{i-1})$. Set $V=V_n$. Then
  \[
  \rank(V)=n+t,
  \]
  where $t$ is the number of full transformation monoids among $T_1,\dots, T_n$.
\end{thm}

The structure of this paper is as follows. In Section~\ref{sec:semi}, we will introduce wreath products of transformation semigroups. We show that wreath products of transformation semigroups are subsemigroups of suitable full transformation monoids, and hence iterated wreath products of transformation semigroups can be defined. In Section~\ref{sec:chap5prf}, we will determine the relative rank of an iterated wreath product of full transformation monoids and symmetric groups modulo an iterated wreath product of symmetric groups and hence prove the main theorem.

\paragraph{Acknowledgements:} The author is funded by the China
Scholarship Council.  

\section{Semigroups and their wreath products}\label{sec:semi}

The definition of wreath products of semigroups can be found in \cite[Chapter 10]{Semiwreath}. To make this paper self-contained, we introduce wreath products below. 

Let $(S, Y), (R, X)$ be transformation semigroups. Denote by $\End{S^X}$ the monoid of endomorphisms on $S^X$. We shall write the action of $\End{S^X}$ on $S^X$ as a left action and the actions of $(S, Y)$ and $(R,X)$ as right actions. Let $\theta:R\to \End{S^X}$ be the homomorphism defined by, for $r\in R$ and $(s_i)_{i\in X}\in S^X$,
\[
(r\theta)(s_i)_{i\in X}=(s_{ir})_{i\in X}.
\]
Suppose that $X=\{1,2,\dots, n\}$, and in this case, the map $r\theta$ on $S^X$ is
\[
(r\theta)(s_1, s_2, \dots, s_n)=(s_{1r}, s_{2r},\dots, s_{nr}).
\]
 Note that the right action of $R$ on the set $X$ yields a left action of $R$ on the semigroup $S^X$. To simplify the notation, we shall use $\leftindex^r(s_i)_{i\in X}$ instead of $(r\theta)(s_i)_{i\in X}$. Let $S\wr_X R$ be the set of elements $(b, r)$ with $b\in S^X$ and $r\in R$. Let $a=(a_i)_{i\in X}, b=(b_i)_{i\in X}\in S^X$ and $r, t\in R$.  For $(a, r), (b,t)\in S\wr _X R$, the multiplication of the semigroup wreath product $S\wr_X R$ is as follows:
\begin{equation}\label{eq:wrsemimul}
(a, r)(b, t)=(a\leftindex^{r}b, rt)=((a_ib_{ir})_{i\in X}, rt).
\end{equation}
Let $(y,x)\in Y\times X$ and $((s_i)_{i\in X}, r)\in S\wr_X R$. The action of $S\wr_X R$ on $Y\times X$ is defined by
\begin{equation}\label{eq:semigroupaction}
(y,x)((s_i)_{i\in X}, r)=(ys_x, xr).
\end{equation}
In this paper, for the wreath product of semigroups $S$ by $R$, if we has specified the set that $R$ acts on, then we may write $S\wr R$ instead, omitting the set.

Assume that $R$ and $S$ are monoids, and we denote by $e$ the identity of a monoid throughout this paper. 
Recall that in a monoid $S$, an element $a$ is a \textit{unit} if there exists $b\in S$ such that $ab=ba=e$.
Suppose that $H\leq S$ is the group of units in $S$ and $G\leq R$ is the group of units in $R$. Then $H\wr G$ can be taken as a wreath product of monoids or a wreath product of groups. Let $(a, f), (b, g)\in H\wr G$ with $a=(a_i)_{i\in X},b=(b_i)_{i\in X}\in H^X$ and $f, g\in G$. When we consider $H\wr G$ as a wreath product of groups, the multiplication is 
\[
(a, f)(b, g)=(ab^{f^{-1}}, fg)=((a_ib_{if})_{i\in X}, fg).
\]
This multiplication coincides with the multiplication when $H\wr G$ is considered as a wreath product of monoids as defined in Equation~\eqref{eq:wrsemimul}. We then give the following result without proof.
\begin{lem}\label{lem:unitsemiwr}
Let $(S, Y), (R, X)$ be transformation monoids. Suppose that $H\leq S$ is the group of units in $S$ and $G\leq R$ is the group of units in $R$. Then $S\wr R$ is a monoid and $H\wr G$ is the group of units in $S\wr R$.
\end{lem}

One can iterate the construction of wreath products above. 
Let $(R,X), (S, Y), (T, Z)$ be transformation semigroups. 
Then $(S\wr R, Y\times X)$ is a transformation semigroup, and $T\wr (S\wr R)$ is the wreath product of $T$ by $(S\wr R, Y\times X)$. 
The iterated wreath product $T\wr (S\wr R)$ is a transformation semigroup on $Z\times Y\times X$. Consequently, if $T_1,\dots, T_n$ are transformation semigroups on $X_1,\dots, X_n$ respectively, then we may iterate the wreath product construction using the induced action. Thus we define $V_1=T_1$ to be the transformation semigroup on $\Omega_1=X_1$. As an inductive hypothesis, for $i\geq2$, we assume that $V_{i-1}$ is a transformation semigroup on $\Omega_{i-1}=X_{i-1}\times\dots\times X_1$. Then let $V_i=T_i\wr V_{i-1}$ be the wreath product of $T_i$ by $(V_{i-1}, \Omega_{i-1})$.

\section{Proof of Theorem~\ref{thm:main}}\label{sec:chap5prf}

We will establish Theorem~\ref{thm:main} in this section. 
We rely on the following key lemmas to establish the theorem. 

\begin{lem}[Ara\'ujo and Schneider~{\cite[Lemma 3.1]{wrSemi}}]\label{lem:ranksemi}
Let $S$ be a finite semigroup and let $G$ be the group of
units in $S$. If $U\subseteq S$ such that $\langle U\rangle= S$, then $\langle U\cap G\rangle = G$. In particular, 
\[
\rank (S) =\rank(S : G)+\rank(G).
\]
\end{lem}

\begin{lem}\label{lem:wrrelarank}
Let $(S, Y), (R, X)$ be transformation monoids. Let $G$ be the group of units of $R$ and let $H$ be the group of units of $S$. If $G$ is transitive on $X$, then
\[
\rank(S\wr R: H\wr G)\leq \rank(R:G)+\rank(S:H).
\]
\end{lem}
\begin{prf}
Lemma~\ref{lem:unitsemiwr} tells us that $H\wr G$ is the group of units of $S\wr R$. Using Lemma~\ref{lem:ranksemi}, in order to prove the lemma, it suffices to verify that 
\[
\rank(S\wr R)\leq \rank(H\wr G)+\rank(R:G)+\rank(S:H).
\]
Set $b=\rank(R:G)$ and $c=\rank(S:H)$. Then there exist $r_1,\dots, r_b\in R$ and $s_1, \dots, s_c\in S$ such that $R=\langle G, r_1,\dots, r_b\rangle$ and $S=\langle H, s_1,\dots,s_c \rangle$. We will show that $H\wr G$ along with $(b+c)$ elements generates $S\wr R$. 

Let $\bar e$ be the identity in $S^X$. Then for $1\leq i\leq b$, let $t_i=(\bar e, r_i)\in S\wr R$. It follows that for any $r\in R$, $(\bar e, r)$ is contained in $T=\langle H\wr G, t_1,\dots, t_b \rangle$. On the other hand, let $x$ be a fixed element of $X$. For each $1\leq i\leq c$, set
\[
\bar s_i=(e,\dots,e,\underset{\mathclap{\substack{\uparrow \\ \text{$x$\nbd coordinate}}}}{s_i},e,\dots,e)\in S^X
\]
and set $u_i=(\bar s_i, e)\in S\wr R$. Then for any $s\in S$, the element
\[
((e,\dots,e,\underset{\mathclap{\substack{\uparrow \\ \text{$x$\nbd coordinate}}}}{s},e,\dots,e), e)
\]
is contained in $U=\langle T, u_1,\dots, u_c\rangle$. As $G$ is a transitive on $X$ and $G$ is contained in $U$, the direct product $S^X$ is contained in $U$ as well. We then deduce that $S\wr R=\langle H\wr G, t_1,\dots, t_b, u_1,\dots, u_c\rangle$, and so the lemma holds.
\end{prf}

Let $n\geq2$ be a positive integer. For each integer $1\leq i\leq n$, let $X_i=\{1, 2, \dots, m_i\}$ where $m_i\geq2$. Let $(T_1, X_1),\dots, (T_n, X_n)$ be transitive transformation monoids on $X_1, \dots, X_n$ respectively such that for $1\leq i\leq n$, the group $G_i$ of units of $T_i$ is transitive on $X_i$ and $\rank(T_i:G_i)\leq1$. Then let $V_1=T_1$ and $\Omega_1=X_1$. For $2\leq i\leq n$, recurssively, let $\Omega_i=X_i\times \Omega_{i-1}$ and 
\[
V_i=T_i\wr V_{i-1}=T_i\wr (T_{i-1}\wr(\dots\wr(T_2\wr T_1)\cdots))
\]
be the  wreath product of the monoid $T_i$ by $(V_{i-1}, \Omega_{i-1})$. According to Lemma~\ref{lem:unitsemiwr}, $V_i$ is a transformation monoid on $\Omega_i$ for $1\leq i\leq n$. We set $V=V_n$ and $\Omega=\Omega_n$. Now let $W_1= G_1$ and let
\[
W_i=G_i\wr W_{i-1}=G_i\wr (G_{i-1}\wr(\dots\wr(G_2\wr G_1)\cdots))
\]
be the iterated wreath product of permutation groups. Set $W=W_n$. 
By repeated use of Lemma~\ref{lem:unitsemiwr}, we deduce the following result.

\begin{cor}\label{cor:unititwr}
Let $V,W$ be defined above. Then $W$ is the group of units in $V$.
\end{cor}

We will use Lemma~\ref{lem:ranksemi} and investigate $\rank(V:W)$ to determine $\rank(V)$. We first determine an upper bound for $\rank(V:W)$.
\begin{lem}\label{lem:semiup}
Let $V,W$ be defined as before. Then
\[
\rank(V:W)\leq \sum_{i=1}^n \rank(T_i:G_i).
\]
\end{lem}
\begin{prf}
The lemma can be proved by induction on $n$. Lemma~\ref{lem:wrrelarank}  establishes the case when $n=2$.
Now let $n>2$ and suppose that the lemma holds for iterated wreath products of transformation monoids involving fewer than $n$ factors. 
Note that we can treat the iterated wreath products $V$ and $W$ as $V=T_n\wr V_{n-1}$ and $W=G_n\wr W_{n-1}$. 
Combining the inductive hypothesis, corollary~\ref{cor:unititwr} and Lemma~\ref{lem:wrrelarank} yields that
\[
\rank(V:W)\leq \rank(T_n:G_n)+\rank(V_{n-1}:W_{n-1})\leq \sum_{i=1}^n \rank(T_i:G_i)
\]
as required.
\end{prf}

In the following, we will determine a 
lower bound for $\rank(V:W)$. 
To this end, we will first present our notation for elements of $V$. 
For $1\leq i\leq n-1$, as $V_{i+1}=T_{i+1}\wr V_{i}$,
an element $v$ in $V_{i+1}$ has the form:
\[
v=((t_j)_{j\in \Omega_{i}}, v_{i})=(\gamma_{i+1}, v_{i})
\]
where $\gamma_{i+1}=(t_j)_{j\in \Omega_{i}}\in T_{i+1}^{\Omega_{i}}$ and $v_{i}\in V_{i}$. Then recursively, we can represent $v$ as:
\[
v=(\gamma_{i+1}, v_{i})=(\gamma_{i+1}, \gamma_{i}, v_{i-1})=\dots=(\gamma_{i+1},\gamma_{i}, \dots,\dots,\gamma_1)
\]
where $\gamma_j\in T_j^{\Omega_{j-1}}$ for $1\leq j\leq i+1$ with $T_1^{\Omega_{0}}=T_1$ by convention. Thus, in the following, we will also write $(\gamma_{i+1},\gamma_{i}, \dots,\dots,\gamma_1)$ for elements in $V_{i+1}$ with $\gamma_{i+1}, \dots, \gamma_1$ as defined above. In particular, for $v_n\in V_n$, we can represent it as $v_n=(\gamma_n,\gamma_{n-1},\dots, \gamma_1)$.

For $1\leq i\leq n-1$, let $\theta_i: V\to V_i$ be the natural projection homomorphism, that is 
\[
(\gamma_n,\gamma_{n-1},\dots, \gamma_1)\theta_i=(\gamma_{i},\gamma_{i-1}, \dots,\dots,\gamma_1)
\]
for any $(\gamma_n,\gamma_{n-1},\dots, \gamma_1)\in V$. Let $\eta_i:\Omega\to\Omega_i$ be the natural projection map. We will then show that $(\theta_i, \eta_i)$ satisfies $(xv)\eta_i=(x\eta_i)(v\theta_i)$ for all $x\in\Omega$ and $v\in V$.

\begin{lem}\label{lem:semihomo}
Let $1\leq i\leq n-1$. For $\eta_i$ and $\theta_i$ as defined above,
\[
(xv)\eta_i=(x\eta_i)(v\theta_i)
\]
for all $x\in\Omega$ and $v\in V$.
\end{lem}
\begin{prf}
For $1\leq i\leq n-1$, let $\pi_i:V_{i+1}\to V_i$ be the natural projection
 semigroup homomorphism, and let $\varphi_i:\Omega_{i+1}\to\Omega_{i}$ 
 be the natural projection. 
 Observe that $\theta_i=\pi_{n-1}\pi_{n-2}\dots\pi_{i}$ and 
 $\eta_i=\varphi_{n-1}\varphi_{n-2}\dots\varphi_{i}$. 
 For $x\in\Omega_{i+1}$ and $v\in V_{i+1}$, 
 by Equation~\eqref{eq:semigroupaction}, $(x v)\varphi_i=(x\varphi_i)(v\pi_i)$.  
 By using this equation repeatedly, we have
 \begin{align*}
 (xv)\eta_i&=(xv)\varphi_{n-1}\varphi_{n-2}\dots\varphi_{i}\\
 &=((x\varphi_{n-1})(v\pi_{n-1}))\varphi_{n-2}\varphi_{n-3}\dots\varphi_{i}\\
 &=((x\varphi_{n-1}\varphi_{n-1})(v\pi_{n-1}\pi_{n-2}))\varphi_{n-3}\varphi_{n-4}\dots\varphi_{i}\\
 &=\dots\\
 &=(x\eta_i)(v\theta_i)
 \end{align*}
 as required.
\end{prf}

Now set $\varepsilon_i=(e,\dots, e)$ to be the identity in $T_i^{\Omega_{i-1}}$ for $1\leq i\leq n$. 
For $1\leq i\leq n$, if $\rank(T_i:G_i)=1$, 
then there exists $a_i\in T_i$ such that 
$T_i=\langle G_i, a_i\rangle$. Such $a_i$ is not a permutation in the group of units.
Now set
\[
\alpha_i=(a_i,e,e,\dots, e)\in T_i^{\Omega_{i-1}},
\]
where $a_i$ occurs in the coordinate with index $(1,1,\dots, 1)\in\Omega_{i-1}$, 
and for any $\omega\in\Omega_{i-1}$, denote by $\alpha_{i,\omega}$ the $\omega$\nbd th entry of $\alpha_i$. Let $\bar \alpha_i$ be the element in $V$ with the form
\[
\bar \alpha_i=(\varepsilon_n, \varepsilon_{n-1},\dots, \varepsilon_{i+1}, \alpha_i,\varepsilon_{i-1}, \dots, \varepsilon_1).
\]
Let $x=(x_n,x_{n-1},\dots, x_1)\in\Omega$ and set $\omega=x\eta_{i-1}=(x_{i-1},x_{i-2},\dots, x_1)$. By using Equation~\eqref{eq:semigroupaction} repeatedly, we deduce
\begin{align}
x\bar\alpha_i&=(x_n,x\eta_{n-1})(\varepsilon_n, \bar\alpha_i\theta_{n-1})\notag\\
 &=(x_n, (x\eta_{n-1})(\bar\alpha_i\theta_{n-1}))\notag\\
 &=(x_n, (x_{n-1}, x\eta_{n-2})(\varepsilon_{n-1}, \bar\alpha_i\theta_{n-2}))\notag\\
 &=\dots\notag\\
 &=(x_n,\dots, x_{i+1}, (x_i, \omega)(\alpha_i, \bar\alpha_i\theta_{i-1}))\notag\\
 &=(x_n,\dots, x_{i+1}, x_i\alpha_{i,\omega}, (x\eta_{i-1})(\bar\alpha_i\theta_{i-1}))\notag\\
 &=\dots\notag\\
 &=(x_n,\dots, x_{i+1}, x_i\alpha_{i,\omega}, x_{i-1},\dots, x_1)\label{eq:chap61}.
\end{align}

Recall that for an element $r$ in a transformation semigroup $(R, X)$, 
\[
\ker(r)=\set{(x,y)\in X\times X}{xr=yr}.
\]
Suppose that two elements $x, y \in \Omega$ satisfy $(x, y)\in\ker(\bar\alpha_i)$. Since $\alpha_{i,\omega}$ is not necessarily a unit, it is possible that $x\neq y$. The following lemma determines the condition for $x=y$.

\begin{lem}\label{lem:kercap}
Let $g\in W$ be a unit in $V$. 
Let $1\leq i, j\leq n$ with $i\neq j$ such that $\rank(T_i:G_i)=\rank(T_j:G_j)=1$. 
If $x,y\in \Omega$ such that $(x,y)\in \ker(\bar\alpha_i)$ and $(xg, yg)\in\ker(\bar\alpha_j)$, then $x=y$. 
\end{lem}
\begin{prf}
Without loss of generality, we assume that $1\leq i<j\leq n$, 
since if $j<i$, then we can set $u=xg$ and $v=yg$, and so $x=ug^{-1}, y=vg^{-1}$. Let $x=(x_n, x_{n-1},\dots, x_1)$ and $y=(y_n, y_{n-1},\dots, y_1)$. 
As $(x,y)\in\ker(\bar\alpha_i)$, Equation~\eqref{eq:chap61} gives $x_k=y_k$ except when $k=i$. 
Using Equation~\eqref{eq:chap61}, the fact that $(xg, yg)\in \ker(\bar\alpha_j)$ yields $(xg)\eta_{j-1}=(yg)\eta_{j-1}$.
Because $g$ is a unit, apply Lemma~\ref{lem:semihomo} to observe that $x\eta_{j-1}=y\eta_{j-1}$. Hence $x_i=y_i$, and the lemma holds.
\end{prf}

Now we can determine a lower bound for $\rank(V:W)$. Ara\'ujo and Schneider determine the relative rank in \cite[Lemma 3.2]{wrSemi} when $n=2$. We adopt their methods and extend their results.
\begin{lem}\label{lem:semilow}
Let $V$ and $W$ be defined as before in this section. Then
\[
\rank(V:W)\geq \sum_{i=1}^n \rank(T_i:G_i).
\]
\end{lem}
\begin{prf}
Set $m=\sum_{i=1}^n \rank(T_i:G_i)$. Suppose that $m=1$. Then $W$ is a proper subsemigroup of $V$, and it follows immediately that $\rank(V:W)\geq1$. 

Suppose that $m\geq2$. We will prove the lemma by contradiction and we assume that $\rank(V:W)< m$. 
Then there exist $\lambda_1,\lambda_2,\dots,\lambda_{m-1}\in V\backslash W$ 
such that $V=\langle W, \lambda_1,\dots,\lambda_{m-1}\rangle$. 
Then let $L=\{\lambda_1,\dots,\lambda_{m-1}\}$. 
For each $1\leq i\leq n$ such that $\rank(T_i:G_i)=1$, $\bar\alpha_i$ is a product of the generating set $W\cup L$. 
Note that each $\bar\alpha_i$ is not a unit. 
Then the product for $\bar\alpha_i$ must involve at least one factor from $L$. Thus, for each $\bar\alpha_i$, there exist $g_i\in W$, $v_i\in V$ and an element $\mu_i\in L$, such that $g_i\mu_iv_i=\bar\alpha_i$.
There then exist $i\neq j$ and $\lambda\in L$ such that
$ g_i\lambda v_i=\bar\alpha_i$ and $g_j\lambda v_j=\bar\alpha_j$. It follows that $\lambda v_i=g_{i}^{-1}\bar\alpha_i$ and $\lambda v_j=g_{j}^{-1}\bar\alpha_j$, and so
\[
\ker(\lambda)\subseteq\ker(g_{i}^{-1}\bar\alpha_i)\cap\ker(g_{j}^{-1}\bar\alpha_j).
\]
Since $\lambda$ is not a unit, there exist different elements $a, b\in \Omega$ such that $(a, b)\in\ker(\lambda)$. 
Let $x=ag_i^{-1}$ and $y=bg_i^{-1}$. 
We have $(x, y)\in\ker(\bar\alpha_i)$, $(xg_ig_j^{-1}, yg_ig_j^{-1})\in\ker(\bar\alpha_j)$ and $x\neq y$. However, Lemma~\ref{lem:kercap} tells us that this is impossible. We deduce that $\rank(V:W)\geq \sum_{i=1}^n \rank(T_i:G_i)$.
\end{prf}

The following lemma follows from Lemma~\ref{lem:semiup} and~\ref{lem:semilow}.

\begin{lem}\label{lem:combine}
For the iterated wreath products of semigroups $V$ and $W$ as defined above,
\[
\rank(V:W)=\sum_{i=1}^n \rank(T_i:G_i).
\] 
\end{lem}

Corollary~\ref{cor:unititwr} tells us that $W$ is the group of units in $V$, and Corollary B in \cite{LuQuick} determines that $d(W)=n$.
Combined with Lemma~\ref{lem:ranksemi}, the main theorem follows immediately.

{\small

\end{document}